\documentclass[a4paper]{amsart}

\usepackage[utf8]{inputenc}
\usepackage{amsmath, amsthm, amssymb}

\input xy
\xyoption{all}

\usepackage{amsmath, amssymb, amsthm}
\usepackage{tikz}
\usepackage{enumerate}

\newtheorem{lem}{Lemma}
\newtheorem{thm}{Theorem}
\newtheorem*{cor}{Corollary}
\newtheorem*{mthm}{Main Theorem}

\newcommand{\tpoly}{T_{\mathrm{poly}}}
\newcommand{\R}{\mathbb{R}}
\newcommand{\grt}{\mathfrak{grt}}
\newcommand{\g}{\mathfrak{g}}
\newcommand{\gc}{\mathrm{GC}}
\newcommand{\MC}{\begin{tikzpicture}[scale=0.08] \tikzstyle{every node}=[circle, draw, fill=black,
                        inner sep=0pt, minimum width=2pt] \path[draw] (0,-1) node{} -- (0,1) node{}; \end{tikzpicture}}
\newcommand{\om}{\Omega(M,\tpoly^\mathrm{vert}(M))}
\newcommand{\tverty}{\tpoly^\mathrm{vert}\vert_{y=0}}

\begin{document}

\title[Globalizing $L_\infty$-automorphisms of the Schouten algebra]{Globalizing $L_\infty$-automorphisms of the \\ Schouten algebra of polyvector fields}
\author[C. Jost]{Christine Jost}
    \address{Stockholm University, Department of
             Mathematics, SE-106 91 Stockholm, Sweden}
    \email{jost@math.su.se}
    \urladdr{http://www.math.su.se/$\sim$jost}

\thanks{}

%
%
%

\begin{abstract}
Recently, Willwacher showed that the Grothendieck-Teichm\"uller group GRT acts by $L_\infty$-automorphisms on the Schouten algebra of polyvector fields $\tpoly (\R^d)$ on affine space $\R^d$. In this article, we prove that a large class of $L_\infty$-automorphisms on $\tpoly (\R^d)$, including Willwacher's, can be globa\-lized. That is, given an $L_\infty$-automorphism of $\tpoly (\R^d)$ and a general smooth manifold $M$ with the choice of a torsion-free connection, we give an explicit construction of an $L_\infty$-automorphism of the Schouten algebra $\tpoly (M)$ on the manifold $M$, depending on the chosen connection. The method we use is the Fedosov trick. 
\end{abstract}

\maketitle

\section{Introduction}

The Grothendieck-Teichm\"uller group is a mysterious group 
defined by Drinfel'd in his study of associators and deformation of Lie algebras \cite{drinfeld1990quasitriangular}. In recent years, connections to objects studied in deformation quantization have become more and more apparent.

In \cite{willwacher2010m} Willwacher proves among other things that the (graded version of the) Grothendieck-Teichm\"uller group GRT acts on the Schouten algebra of polyvector fields $\tpoly(\R^d)$ on affine space. In this article, we extend Willwacher's result to the Schouten algebra of polyvector fields $\tpoly(M)$ on a general smooth manifold $M$. In fact, we prove that any $L_\infty$-automorphism of $\tpoly(\R^d)$ satisfying certain conditions can be globalized. The method we use goes back to Dolgushev's globalization \cite{dolgushev2005covariant} of Kontsevich's deformation quantization \cite{kontsevich2003deformation}. Dolgushev's result in turn uses the famous Fedosov trick \cite{fedosov1994simple}. Similar methods have also been used in \cite{cattaneo2001globalization} and \cite{cattaneo2002local}.
The main result is the following:
\begin{mthm}
 Let $F$ be an $L_\infty$-automorphism of the Schouten algebra $\tpoly(\R^d)$ on affine space satisfying the following conditions:
\setcounter{enumi}{0} 
\begin{enumerate}[(1)] 
 \item $F$ extends to an $L_\infty$-automorphism of $\tpoly(\R^d_\mathrm{formal})$, the Schouten algebra on the ``fat point'' $\R^d_\mathrm{formal} = \mathrm{Spec}\ \R[[x_1, \ldots, x_d]]$. \label{Rformal} 
 \item The extension of $F$ to $\tpoly(\R^d_\mathrm{formal})$ is invariant under linear change of coordinates of $\R^d_\mathrm{formal}$. \label{changeofcoordinates}
 \item $F_1 = \mathrm{id}$, and for $n\ge 2$, the $n$-ary part $F_n$ of $F$ vanishes on vector fields. That means $F_n(v_1, \ldots, v_n)=0$ for vector fields $v_1, \ldots, v_n$. \label{vectorfields}
\item $F$ vanishes if one of the inputs is a vector field that is linear in the standard coordinates on $\R^d$. That means $F_n(\gamma_1, \ldots, A_j^ix^j\frac{\partial}{\partial x^i}, \ldots, \gamma_n) = 0$ for arbitrary polyvector fields $\gamma_1, \ldots , \gamma_n$ and a vector field $A_j^ix^j\frac{\partial}{\partial x^i}$.  \label{linearvectorfield}
\end{enumerate}
Let $M$ be a smooth manifold with a torsion-free connection. Then the construction described in the proof yields an $L_\infty$-automorphism $F^\mathrm{glob}$ of the Schouten algebra $\tpoly(M)$ on $M$, depending on $F$ and the chosen connection. 
\end{mthm}

 In \cite{willwacher2010m} Willwacher constructs a Lie algebra isomorphism from the Grothendieck-Teichm\"uller algebra $\grt$ to the zeroth cohomology $H^0(\mathrm{GC}_2)$ of Kontsevich's graph complex $\mathrm{GC}_2$. The graph complex in turn acts on the Schouten algebra $\tpoly(\R^d)$ on affine space by $L_\infty$-automorphisms defined up to homotopy equivalence. Using our main theorem, we prove that these automorphisms can be used to construct $L_\infty$-automorphisms of the Schouten algebra $\tpoly(M)$ on any smooth manifold  $M$. The essential detail in Willwacher's work is that any element of the Grothendieck-Teichm\"uller algebra can be represented by a graph all of whose vertices are at least trivalent. From this fact it follows that Willwacher's automorphisms satisfy the third and forth condition in our main theorem. 
\begin{cor}
 The Grothendieck-Teichm\"uller group GRT acts on the Schouten algebra $\tpoly(M)$ on a general smooth manifold $M$ by  $L_\infty$-automorphisms defined up to homotopy equivalence. The action depends on the choice of a torsion-free connection on $M$.
\end{cor}

Note that both the main theorem and its corollary depend on the choice of a connection on $M$. The ideas in Appendix C of \cite{bursztyn2012morita} might be used to prove that the homotopy class of the $L_\infty$-automorphism does not depend on the choice of the connection.

 We give a short outline of the content of this article: In Section \ref{vertpolyvectors}, we construct the basic objects used for the globalization, vertical polyvector fields and differential forms with values in them. In Section \ref{fedosovres}, we construct a non-trivial section of  the vector bundle whose sections are differential forms with values in vertical polyvector fields. The construction depends on the choice of a connection on $M$. This so-called Fedosov trick yields a resolution of $\tpoly(M)$ as a Lie algebra. In Section \ref{proofmainthm}, we prove the main theorem constructing an $L_\infty$-automorphism of $\tpoly(M)$ from an automorphism of $\tpoly(\R^d)$ using the Fedosov resolution. In the last section, we first recall Willwacher's action of GRT on $\tpoly(\R^d)$ and finally prove the corollary, i.e., that this action can be globalized using our main theorem.

We use the Einstein summation convention. The degree of an element $f$ of a graded vector space is denoted by $\lvert f \rvert$.

\section{Vertical polyvector fields}\label{vertpolyvectors}

In this section, we present the basic constructions needed in this article, namely vertical polyvector fields, differential forms with values in them and the vertical Schouten bracket. The aim is to construct a large vector bundle on a smooth manifold $M$ such that the fiber of each point is isomorphic as a vector space to the Schouten algebra $\tpoly(\R^d_\mathrm{formal})$ of the ``fat point'' $\R^d_\mathrm{formal}$, the formal completion of affine space $\R^d$ at the origin. The reason we have to use $\R^d_\mathrm{formal}$ instead of $\R^d$ is to ensure convergence of certain recurrence equations later. We will often work with local trivializations of the different vector bundles we consider.

Let $M$ be a $d$-dimensional smooth real manifold. We start by recalling the definition of the usual Schouten algebra $\tpoly(M)$ of polyvector fields on $M$.
The underlying set of the Schouten algebra is the exterior algebra over the $C^\infty(M)$-module of vector fields, i.e. $T_\mathrm{poly}(M) = \bigoplus_{q=0}^d \Lambda^q(\Gamma(TM))$.  
The usual Lie bracket on vector fields extends by the graded Leibniz rule to an odd graded Lie bracket on the exterior algebra, called Schouten bracket. On a local chart $U \subset M$, the underlying graded $\mathbb{R}$-vector space of the Schouten algebra $\tpoly(U)$ is isomorphic to $C^\infty(U)[\varphi_1, \ldots, \varphi_d]$, writing $\varphi_i$ for $\frac{\partial}{\partial x^i}$ and with the grading $\lvert \varphi_i \rvert = 1$. The Schouten bracket $[-,-]^S$ is on this local chart given by the formula
\begin{equation*}
 [f,g]^\mathrm{S} = (-1)^{\lvert f \rvert} \frac{\partial f}{\partial x^i}\frac{\partial g}{\partial \varphi_i} + (-1)^{\lvert f \rvert \lvert g \rvert + \lvert g \rvert} \frac{\partial g}{\partial x^i} \frac{\partial f}{\partial \varphi_i}
\end{equation*}
for $f,g \in \tpoly(U)$.

We continue with the definition of vertical polyvector fields. Let $TM$ be the total space of the tangent bundle of $M$. The projection $\pi: TM \rightarrow M$ induces the differential $\mathrm{d}\pi: T_{TM} \rightarrow \pi^* T_M$. We are interested in the kernel of $\mathrm{d}\pi$, a vector bundle over $TM$. To describe it in a more detailed way, we observe that a local chart of $TM$ is isomorphic to an open subset of $\R^{2d}$. 
It is possible to choose the local coordinate system on $TM$ in such a way that the projection $\pi$ to $M$ is given by $x^i \mapsto x^i$ and $y^i \mapsto 0$. We obtain a local description of sections of $\mathrm{ker}\ \mathrm{d}\pi$,
\begin{equation*}
 \Gamma(U, \mathrm{ker}\ \mathrm{d}\pi) = \{f^i(x,y) \frac{\partial}{\partial y^i}, f^i(x,y) \mbox{ smooth} \}.
\end{equation*}
Global sections of $\mathrm{ker}\ \mathrm{d}\pi$ are what is usually called vertical vector fields. However, we will need a slightly different definition allowing formal power series in the $y$s. Define the $C^\infty(U)[[y^1, \ldots, y^d]]$-module of vertical vector fields on $U \subset M$ by
\begin{equation*}
 T^\mathrm{vert}(U) = \{\sum_{I \in \mathbb{N}^d}f^i_I(x)y^I \frac{\partial}{\partial y^i}, f^i_I(x) \mbox{ smooth} \}.
\end{equation*}
One checks that these glue together to a $C^\infty(M)$-module $T^\mathrm{vert}(M)$. Besides the $C^\infty(M)$-module structure, $T^\mathrm{vert}(M)$ also is a $C^\infty(M)[[y^1, \ldots, y^d]]$-module, where $C^\infty(M)[[y^1, \ldots, y^d]]$ is glued together from the rings $C^\infty(U)[[y^1, \ldots, y^d]]$. The $C^\infty(M)$-module of vertical polyvector fields is then defined as the exterior algebra of $T^\mathrm{vert}(M)$ over $C^{\infty}(M)[[y^1, \ldots, y^d]]$, i.e.,
\begin{equation*}
 \tpoly^\mathrm{vert}(M) = \bigoplus_{q=0}^d \Lambda^q_{C^\infty(M)[[y^1, \ldots, y^d]]} T^\mathrm{vert}(M).
\end{equation*}
Furthermore, differential forms with values in vertical polyvector fields are defined by
\begin{equation*}
 \Omega(M, \tpoly^\mathrm{vert}(M)) = \tpoly^\mathrm{vert}(M) \otimes_{C^\infty(M)} \Omega M,
\end{equation*}
where $\Omega M$ denotes as usual the de Rham algebra of differential forms on $M$.
We work on a local chart $U$ of $TM$, denoting coordinate functions by $x^1, \ldots, x^d, y^1, \ldots, y^d$ as before.
Writing $\frac{\partial}{\partial y^i} = \psi_i$ and $\mathrm{d}x^i = \eta^i$, differential forms with values in vertical polyvector fields are then elements of 
\begin{equation*}
 C^\infty(U)[[y^1, \ldots, y^d]][\psi_1, \ldots, \psi_d,\eta^1, \ldots, \eta^d].
\end{equation*}
Setting the degrees $\lvert y^i \rvert = 0$ and $\lvert \psi_i \rvert = \lvert \eta^i \rvert = 1$, we obtain a grading on $\om$. Denote by $\Omega^r(M, \tpoly^\mathrm{vert}(M))$ the subspace of elements of degree $r$. Putting in another way, we have
\begin{equation*}
 \Omega^r(M, \tpoly^\mathrm{vert}(M)) = \bigoplus_{p+q=r} \tpoly^\mathrm{vert,p} \otimes \Gamma(\Lambda^q(T^*M)),
\end{equation*}
where the elements of $\tpoly^\mathrm{vert,p}$ have degree $p$ with respect to the $\frac{\partial}{\partial y^i}$.
Define the vertical Schouten bracket on $\om$ locally by
\begin{equation*}
 [f,g]^\mathrm{vert} = (-1)^{\lvert f \rvert} \frac{\partial f}{\partial y^i}\frac{\partial g}{\partial \psi_i} + (-1)^{\lvert f \rvert \lvert g \rvert + \lvert g \rvert} \frac{\partial g}{\partial y^i} \frac{\partial f}{\partial \psi_i}
\end{equation*}
for $f,g \in \Omega(U,\tpoly^\mathrm{vert}(U))$. One checks that the definition is independent of the choice of the local chart. As the Schouten bracket, the vertical Schouten bracket has degree -1 and provides $\om$ with the structure of an odd graded Lie algebra.

We also introduce the sub-$C^\infty(M)$-module $\tverty (M)$ of vertical polyvector fields which are constant with respect to the $y$. On a local chart $U$, we have that $\tverty(U)$ is isomorphic as a $C^\infty(U)$-module to $C^\infty(U)[\psi_1, \ldots, \psi_d]$.
One sees easily that $\tverty (M)$ and $\tpoly (M)$ are isomorphic as $C^\infty (M)$-modules.

We conclude this section with the first lemma on vertical polyvector field, showing that $\om$ is a resolution of $\tpoly (M)$ as a $C^\infty (M)$-module. Define
\begin{equation*}
 \delta:  \Omega^r(M, \tpoly^\mathrm{vert}(M)) \rightarrow  \Omega^{r+1}(M, \tpoly^\mathrm{vert}(M))
\end{equation*}
locally by $\delta = \mathrm{d}x^i \frac{\partial}{\partial y^i}$. Furthermore, define a contracting homotopy
\begin{equation*}
 \delta^*:  \Omega^r(M, \tpoly^\mathrm{vert}(M)) \rightarrow  \Omega^{r-1}(M, \tpoly^\mathrm{vert}(M))
\end{equation*}
by
\begin{equation*}
 \delta^*(f_{I,J}(x,\psi)y^I\eta^J) = \frac{1}{p+q} y^a \frac{\partial}{\partial \eta^a} f(x,\psi)y^I\eta^J,
\end{equation*}
where $I$ and $J$ are multi-indices of total degree $p$ and $q$, respectively. For $f=f(x,\psi)y^0\eta^0$, we set $\delta^*(f) = 0$. Furthermore, define the projection
\begin{equation*}
 \sigma: \om \rightarrow \tverty(M) \subset \om
\end{equation*}
locally by $\sigma(\eta^i) = \sigma (y^i) = 0$ and $\sigma(\psi_i) = \psi_i$.
\begin{lem}
For any $f \in \om$, it holds that
\begin{equation} \label{contractinghomotopy}
 f = \sigma f + \delta\delta^* f + \delta^* \delta f,
\end{equation}
hence 
 \begin{itemize}
  \item[(i)] $H^n( \om , \delta)  = 0$ for $n \neq 0$,
  \item[(ii)] $H^0( \om ,\delta ) \cong \tpoly(M)$.
 \end{itemize}
\end{lem}
\begin{proof}
 Equation \eqref{contractinghomotopy} is easy to check. The remainder of the lemma follows because $\tverty (M)$ and $\tpoly (M)$ are isomorphic as $C^\infty(M)$-modules.
\end{proof}

\section{The Fedosov resolution}\label{fedosovres}

It is possible to see $\om$ as sections of a vector bundle over the total space of a vector bundle whose sections form $\tverty (M)$. In the previous lemma, we have showed that the zero section of that bundle yields a resolution of $\tpoly (M) \cong \tverty (M)$ as a $C^\infty (M)$-module. Now we will construct a very non-trivial section that yields a resolution of $\tpoly (M)$ as a Lie algebra.
It is called the Fedosov resolution, as it uses the Fedosov trick \cite{fedosov1994simple}  of using a differential depending on the choice of a torsion-free connection. Besides Fedosov's work, we also rely very much on Dolgushev's \cite{dolgushev2005covariant}. The first two lemmas are taken directly from \cite{dolgushev2005covariant} and are included mostly for self-containedness.

We will work on a local chart of $M$ from now on. All local formulas are independent of the choice of the chart, if not said otherwise.

Choose a torsion-free connection on the manifold $M$ and denote its Christoffel symbols by $\Gamma_{ij}^k(x)$. Define a derivation $\nabla$ on $\om$ as $\nabla = d + [\Gamma,-]^\mathrm{vert}$, where $d = \mathrm{d}x^i \frac{\partial}{\partial x^i}$ and $\Gamma = -\mathrm{d}x^i\Gamma_{ij}^k(x)y^j\frac{\partial}{\partial y^k}$. This derivation is not a differential. However, it holds that
\begin{itemize}
 \item $\nabla^2 = [R,-]^\mathrm{vert}$, where $R= -\frac{1}{2}\mathrm{d}x^i\mathrm{d}x^jR_{kij}^l(x) y^k \frac{\partial}{\partial y^l}$ is given by the Riemann curvature tensor of the connection, and
\item $\delta \nabla + \nabla \delta = 0$.
\end{itemize}
\begin{lem}
 There exists an element $A$ in $\om$ such that $\delta^*A=0$ and 
\begin{equation*}
 D := \nabla - \delta + [A,-]^\mathrm{vert}
\end{equation*}
is a differential, where $A$ has the form $A = \sum_{p=2}^\infty \mathrm{d}x^k A_{k,i_1\ldots i_p}^j(x) y^{i_1} \ldots y^{i_p} \frac{\partial}{\partial y^j}$.
\end{lem}
This and the following lemma use the same technique of proof, which we explain in detail at first and use in a more sketchy way later. The main idea is to use the contracting homotopy equation \eqref{contractinghomotopy} to produce a recursive equation.
\begin{proof}[Proof of Lemma 2]
 We have to show that $D^2 = 0$ or, equivalently, that 
\begin{equation}\label{eqA}
R+\nabla A + \frac{1}{2}[A,A]^\mathrm{vert} = \delta A.
\end{equation}
As we are looking for an $A$ such that $\delta^* A=\sigma A = 0$, we insert these identities into the contracting homotopy equation \eqref{contractinghomotopy} and get that $A$ should suffice $A = \delta^* \delta A$. Inserting the left hand side of \eqref{eqA} for $\delta A$ yields the following recursion formula for $A$:
\begin{equation*}
 A = \delta^* R + \delta^*(\nabla A + \frac{1}{2}[A,A]^\mathrm{vert}).
\end{equation*}
The recursion converges because $\delta^*$ increases the degree in the $y$s.

We still have to prove that $A$ actually satisfies $C := R + \nabla A + \frac{1}{2}[A,A]^\mathrm{vert} - \delta A = 0$. Observe that it follows from the Bianchi identities for the Riemann curvature tensor that $\delta R = \nabla R = 0$. Using this, we get that $\delta C = \nabla C + [A,C]^\mathrm{vert}$. Furthermore, it holds that $\sigma C = \delta^* C = 0$. Inserting these equations for $\delta C$, $\delta^* C$ and $\sigma C$ into the contracting homotopy equation \eqref{contractinghomotopy}, we get the recursive equation
\begin{equation*}
 C = \delta^*(\nabla C + [A,C]^\mathrm{vert}).
\end{equation*}
As $\delta^*$ increases the degree in the $y$s, this equation has the unique solution $C=0$, which concludes the proof.
\end{proof}
We are now ready to show that the differential $D$ that we have just constructed still yields a resolution of $\tpoly (M)$ as a $C^\infty (M)$-module.
\begin{lem}
 \begin{itemize}
  \item[(i)] $H^n( \om , D)  = 0$ for $n \neq 0$
  \item[(ii)] $H^0( \om ,D ) \cong \tpoly(M)$
 \end{itemize}
\end{lem}
\begin{proof}
We start by proving part (i). Let $f$ be a cocycle of degree $\ge 1$, i.e., $f \in \Omega^r(M,\tpoly^\mathrm{vert}(M))$, $r\ge 1$, $Df = 0$. We need to find a $g \in \Omega^{r-1}(M,\tpoly^\mathrm{vert}(M))$ such that $f$ is the image of $g$, i.e., $Dg = f$. We restrict our search to such $g$ that satisfy $\sigma g = \delta^* g = 0$. For these $g$ it holds that $f=Dg = \nabla g - \delta g + [A,g]^\mathrm{vert}$. Furthermore, by the contracting homotopy equation \eqref{contractinghomotopy}, it holds that $g = \delta^* \delta g$. Inserting, we get the recursive equation
\begin{equation*}
 g = -\delta^*f + \delta^*(\nabla g + [A,g]^\mathrm{vert}).
\end{equation*}
The convergence follows as usual from the fact that $\delta^*$ increases the degree in the $y$s.

We show that in fact $h:=Dg-f = 0$. One sees that $\delta^* h = \sigma h = 0$ and $Dh=0$. Hence we get the same recursion equation as for $C$ in the foregoing lemma,
\begin{equation*}
 h = \delta^*(\nabla h + [A,h]^\mathrm{vert}),
\end{equation*}
which has the unique solution $h=0$.

Now we prove part (ii). The aim is to find a suitable section $\tau: \tverty (M) \rightarrow \om$. Let $f_0 \in \tverty (M)$, find a unique $f\in \Omega^0(M,\tpoly^\mathrm{vert}(M))$ such that $Df=0$ and $\sigma f = f_0$. As $f$ has degree 0 it follows that $\delta^* f = 0$. Together with $Df = 0$ this yields the recursive equation
\begin{equation*}
 f = f_0 + \delta^*( \nabla f + [A,f]^\mathrm{vert})
\end{equation*}
with convergence as usual.

We have to show that actually $u:=Df = 0$. As $\sigma u = \delta^* u = 0$ and $Du = 0$, this follows in the same way as for $g$ in the first part of the proof.
\end{proof}
As final step for the Fedosov resolution, we prove now that the Fedosov resolution of $\tpoly (M)$ indeed is a resolution as Lie algebra.
\begin{lem}
The induced Lie algebra structure on $\tpoly (M)$ induced by the vertical Schouten bracket on  $\om$ is isomorphic to the Schouten bracket. 
\end{lem}

\begin{proof}
To simplify notation, we will identify $\tpoly (M)$ and $\tverty (M)$ in this proof. Especially, we will write $\sigma$ for the composition
\begin{equation*}
 \Omega(M,T_{\mathrm{poly}}^{\mathrm{vert}}(M)) \overset{\sigma}{\longrightarrow} T_{\mathrm{poly}}^{\mathrm{vert}}(M)|_{y=0} \overset{\sim}{\longrightarrow} T_{\mathrm{poly}}(M)
\end{equation*}
and $\tau$ for the composition
\begin{equation*}
 \Omega(M,T_{\mathrm{poly}}^{\mathrm{vert}}(M)) \overset{\tau}{\longleftarrow} T_{\mathrm{poly}}^{\mathrm{vert}}(M)|_{y=0} \overset{\sim}{\longleftarrow} T_{\mathrm{poly}}(M).
\end{equation*}

We have to show that 
 $ \tau[f_0,g_0]^S = [\tau f_0,\tau g_0]^{\mathrm{vert}}$
 for $f_0,g_0 \in T_{\mathrm{poly}}(M)$. 
By the definition of $\tau$ in the proof of the previous lemma, this is equivalent to showing that $[\sigma f, \sigma g]^S = \sigma[f,g]^\mathrm{vert}$ for $f,g \in \Omega^0(M,T_{\mathrm{poly}}^{\mathrm{vert}}(M))$ such that $Df=Dg=0$.

From $Df=0$ it follows that
\begin{equation*}
 \mathrm{d}x^i \frac{\partial f}{\partial y^i} =  \mathrm{d}x^i \frac{\partial f}{\partial x^i} -  \mathrm{d}x^i \Gamma_{ij}^k(x)\psi_k\frac{\partial f}{\partial \psi_j} + \mathrm{d}x^i (\mbox{terms containing }y).
\end{equation*}
Hence, using that $f$ lies in $\Omega^0(M,T_{\mathrm{poly}}^{\mathrm{vert}}(M))$, we have
 $\sigma(\frac{\partial f}{\partial y^i} ) = \sigma( \frac{\partial f}{\partial x^i}) - \sigma( \Gamma_{ij}^k(x)\psi_k\frac{\partial f}{\partial \psi_j})$.
From the explicit formula for the vertical Schouten bracket, we obtain
\begin{equation*}
\sigma[f,g]^{\mathrm{vert}} = [\sigma f,\sigma g]^S  - \sigma S,
\end{equation*}
where
\begin{equation*}
 S 	= (-1)^{\lvert f \rvert} \Gamma_{ij}^k(x)\psi_k\frac{\partial  f}{\partial \psi_j}\frac{\partial  g}{\partial \psi_i} 
		+ (-1)^{\lvert  f \rvert \lvert g \rvert + \lvert g \rvert} \Gamma_{ij}^k(x)\psi_k\frac{\partial  g}{\partial \psi_j}\frac{\partial  f}{\partial \psi_i},
\end{equation*}
which is zero because $\Gamma_{ij}^k(x)$ is symmetric in the lower indices. This proves the claim.
\end{proof}

\section{Proof of the main theorem}\label{proofmainthm}

In this section, we describe a construction that globalizes $L_\infty$-automorphisms of the Schouten algebra if they satisfy certain conditions. The construction depends on the choice of a torsion-free connection on the manifold $M$. The theorem and its proof are analogous to Proposition 3 in Dolgushev's article \cite{dolgushev2005covariant}.
\begin{mthm}
 Let $F$ be an $L_\infty$-automorphism of the Schouten algebra $\tpoly(\R^d)$ on affine space satisfying the following conditions:
\setcounter{enumi}{0} 
\begin{enumerate}[(1)] 
 \item $F$ extends to an $L_\infty$-automorphism of $\tpoly(\R^d_\mathrm{formal})$, the Schouten algebra on the ``fat point'' $\R^d_\mathrm{formal} = \mathrm{Spec}\ \R [[x_1, \ldots, x_d]]$.
 \item The extension of $F$ to $\tpoly(\R^d_\mathrm{formal})$ is invariant under linear change of coordinates of $\R^d_\mathrm{formal}$.
 \item $F_1 = \mathrm{id}$, and for $n\ge 2$, the $n$-ary part $F_n$ of $F$ vanishes on vector fields. That means $F_n(v_1, \ldots, v_n)=0$ for vector fields $v_1, \ldots, v_n$. 
\item $F$ vanishes if one of the inputs is a vector field that is linear in the standard coordinates on $\R^d$. That means $F_n(\gamma_1, \ldots, A_j^ix^j\frac{\partial}{\partial x^i}, \ldots, \gamma_n) = 0$ for arbitrary polyvector fields $\gamma_1, \ldots ,\gamma_n$ and a vector field $A_j^ix^j\frac{\partial}{\partial x^i}$. 
\end{enumerate}
Let $M$ be a smooth manifold and choose a torsion-free connection on $M$. Then the construction below yields an $L_\infty$-automorphism $F^\mathrm{glob}$ of the Schouten algebra $\tpoly(M)$ on $M$, depending on $F$ and the chosen connection. 
\end{mthm}

The construction of the global $L_\infty$-morphism consists of three steps. At first, we construct an $L_\infty$-automorphism of $\Omega(M,\tpoly^\mathrm{vert}(M))$. In the second step, we twist this morphism with a suitable Maurer-Cartan element to yield an $L_\infty$-morphism of $\Omega(M,\tpoly^\mathrm{vert}(M))$ commuting with the differential $D$. In the third step, this will induce an $L_\infty$-automorphism on $H^0(\Omega(M,\tpoly^\mathrm{vert}(M))) \cong \tpoly(M)$.

\subsection*{First step}
Choose an open chart $U$ of the manifold $M$ such that $\Omega(U,\tpoly^\mathrm{vert}(U))$ trivializes to 
\begin{equation*}
 C^\infty(U)[[y^1,\ldots, y^d]][\psi_1,\ldots,\psi_d,\eta^1,\ldots, \eta^d],
\end{equation*}
as described in section \ref{vertpolyvectors}. Because of condition (\ref{Rformal}), there is an extension of $F$ to an $L_\infty$-automorphism $F^\mathrm{formal}$ of 
\begin{equation*}
 \tpoly(\R^d_\mathrm{formal}) \cong \R[[x^1, \ldots, x^d]][\varphi_1,\ldots, \varphi_d]
\cong \R[[y^1, \ldots, y^d]][\psi_1,\ldots, \psi_d].
\end{equation*}
Any element of $\Omega(U,\tpoly^\mathrm{vert}(U))$ can be written as a (possibly infinite) linear combination \begin{equation*}
     \sum_{I,J \in \mathbb{N}^d} f(x^1, \ldots, x^d, \eta^1, \ldots, \eta^d) y^I \psi_J,                                                                                 
\end{equation*}
where $f \in C^\infty(U)[\eta^1, \ldots, \eta^d]$. By $C^\infty(U)[\eta^1, \ldots, \eta^d]$-linear continuation, $F^\mathrm{formal}$ induces an $L_\infty$-automorphism $F^\mathrm{vert}\vert_U$ of $\Omega(U,\tpoly^\mathrm{vert}(U))$. By condition (\ref{changeofcoordinates}), the construction is independent of the choice of $U$. Hence it yields an $L_\infty$-automorphism $F^\mathrm{vert}$ of $\Omega(M,\tpoly^\mathrm{vert}(M))$ with the vertical Schouten bracket.

\subsection*{Second step}

By construction, the automorphism $F^\mathrm{vert}$ commutes with the differential $d = \mathrm{d}x^i\frac{\partial}{\partial x^i}$. In general, however, it does not commute with
\begin{equation*}
 D = d + [\Gamma - \mathrm{d}x^i\frac{\partial}{\partial y^i} + A,-]^{\mathrm{vert}}.
\end{equation*}
We denote $\Gamma - \mathrm{d}x^i\frac{\partial}{\partial y^i} + A$ by $B$ and obtain the compact description $D = d + [B,-]^\mathrm{vert}$.

Before continuing with the proof, we recall some well-known facts about Maurer-Cartan elements. Let 
\begin{equation*}
 \Phi: (\g,[-,-]_\g,d_\g) \rightarrow (\g',[-,-]_{\g'},d_{\g'})
\end{equation*}
 be an $L_\infty$-morphism of dg-Lie algebras and $\pi$ a Maurer-Cartan element of $\g$. Then $\pi' = \sum_{i=1}^\infty \frac{1}{i!} \Phi_i(\pi^i)$ is a Maurer-Cartan element of $\g'$. It follows that $\g_\pi = (\g,[-,-]_\g,d_\g + [\pi,-]_\g )$ and $\g'_{\pi'}=(\g',[-,-]_{\g'},d_{\g'} + [\pi',-]_{\g'})$ are dg-Lie algebras. Furthermore, it holds that $\Phi_\pi = \mathrm{exp}(-\pi') \circ \Phi \circ \mathrm{exp}(\pi)$ is an $L_\infty$-morphism from $\g_\pi$ to $\g'_{\pi'}$, where $\mathrm{exp}(\pi)(T) = \sum_{i=0}^\infty \frac{1}{i!} \pi^i \cdot T $ and $\mathrm{exp}(-\pi')$ is defined analogously.

We apply this construction to our situation, where $F^\mathrm{vert}\vert_U$ is an $L_\infty$-automorphism of $ (\Omega(U,\tpoly^\mathrm{vert}(U)),[-,-]^\mathrm{vert},d)$ and $B$ is a Maurer-Cartan element. From condition (\ref{vectorfields}), it follows that $\sum_{i=1}^\infty \frac{1}{i!} (F^\mathrm{vert}\vert_U)_i(B^i) = B$, as $B$ is a vertical vector field. Hence it follows that $(F^\mathrm{vert}\vert_U)_B = \mathrm{exp}(-B) \circ F^\mathrm{vert}\vert_U \circ \mathrm{exp}(B)$ is an $L_\infty$-automorphism of $(\Omega(U,\tpoly^\mathrm{vert}(U)),[-,-]^\mathrm{vert},D)$, because $D = d + [B,-]^\mathrm{vert}$.

We have worked on the local chart $U$ so far, as the formula for $B$ is dependent on the choice of local coordinates. It will become clear, however, that the definition of $(F^\mathrm{vert}\vert_U)_B$ is not. 
We analyze how $B$ transforms under change of coordinates. The terms $\mathrm{d}x^i\frac{\partial}{\partial y^i}$ and $A$ are invariant under change of coordinates. The transformation of $\Gamma$ is more complicated due to the presence of the Christoffel symbols. We compute that $B$ transforms as $B' = B + \mathrm{d}x^i H_{ij}^k(x) y^j \frac{\partial}{\partial y^k}$ for some $H_{ij}^k(x)$, where the exact form of $H_{ij}^k(x)$ is not important.\footnote{This step is taken directly from \cite{dolgushev2005covariant}, see Equation (58).}
We take a closer look at the explicit formula for $(F^\mathrm{vert}\vert_U)_B$. It holds that $(F^\mathrm{vert}\vert_U)_{B,n} = \sum_{i=1}^\infty \frac{1}{i!} (F^\mathrm{vert}\vert_U)_{n+i}( B^i -)$. 
However, $F^\mathrm{vert}\vert_U$ is zero on any summand of the form $\mathrm{d}x^i H_{ij}^k(x) y^j \frac{\partial}{\partial y^k}$ by condition (\ref{linearvectorfield}). 
Hence the definition of $(F^\mathrm{vert}\vert_U)_B$ is independent of the choice of $U$ and thus glues to an $L_\infty$-automorphism $F^{\mathrm{vert},D}$ of $\om$ commuting with $D$.

\subsection*{Third step}
As $F^{\mathrm{vert},D}$ constructed in the last step commutes with the differential $D$ of $\Omega(M,\tpoly^\mathrm{vert}(M))$, it induces an $L_\infty$-automorphism on cohomology. As $H^0(\Omega(M,\tpoly^\mathrm{vert}(M))) \cong \tpoly(M)$, it hence follows that this is an $L_\infty$-automorphism $F^\mathrm{glob}$ on $\tpoly(M)$ depending on $F$ and the choice of a connection from which the differential $D$ is constructed. This is the last step of the construction of the global $L_\infty$-automorphism $F^\mathrm{glob}$. As we have checked that each step of the construction works if the given automorphism $F$ satisfies the conditions given in the theorem, this also concludes the proof of the main theorem.

\section{Proof of the corollary}\label{proofcorollary}

We start by stating the theorem of Willwacher's that we use.
\begin{thm} (Willwacher \cite{willwacher2010m}) The Grothendieck-Teichm\"uller group GRT acts on the Schouten algebra $\tpoly(\R^d)$ on affine space by $L_\infty$-automorphisms.
\end{thm}
 As we will need some details from the proof of this theorem later, we include the proof here. The way we present it here is analogous to Section 3 in \cite{merkulov2010grothendieck}. We start by recalling the definitions and constructions needed for the proof: the Lie algebra associated to an operad, Kontsevich's graph complex, and the Chevalley-Eilenberg cochain complex.

Let $\mathcal{P} = \{\mathcal{P}(n)\}_{n\in \mathbb{N}}$ be an operad with partial composition $\circ_i$. Following \cite{kapranov2001modules}, define a Lie algebra structure on the graded vector space $\prod_{n=1}^\infty \mathcal{P}(n)$ by
\begin{equation*}
 [\mu , \nu ]= \sum_{i=1}^m \mu \circ_i \nu - (-1)^{\lvert \mu \rvert \lvert \nu \rvert} \sum_{i=1}^n \nu \circ_i \mu
\end{equation*}
for $\mu \in \mathcal{P}(m)$ and $\nu \in \mathcal{P}(n)$. This Lie bracket induces a Lie algebra structure on the space of $\mathbb{S}$-coinvariants $\prod_{k=1}^\infty \mathcal{P}(k)_{S_k}$. As we work in characteristic 0, the space of $\mathbb{S}$-coinvariants and the space of $\mathbb{S}$-invariants are isomorphic as vector spaces. By this isomorphism, the Lie bracket on $\prod_{n=1}^\infty \mathcal{P}(n)$ also induces a Lie algebra structure on the space of $\mathbb{S}$-invariants $\prod_{k=1}^\infty \mathcal{P}(k)^{S_k}$.

We continue by defining the operad $\mathsf{Gra}$ and graph complex $\mathrm{GC}_2$. Let $\mathrm{gra}_{n,k}$ be the set of all undirected graphs with $n$ numbered vertices and $k$ numbered edges. The symmetric group $S_k$ acts on $\mathrm{gra}_{n,k}$ by renumbering of the edges. Let $\mathrm{sgn}_k$ be the sign representation of  $S_k$. Then define 
\begin{equation*}
 \mathsf{Gra}(n) = \bigoplus_{k\ge 0} \left( \R \langle \mathrm{gra}_{n,k} \rangle \otimes_{S_k} \mathrm{sgn}_k \right) [k].
\end{equation*}
The vector spaces $\mathsf{Gra}(n)$ admit an action of $S_n$ by renumbering of the vertices. Furthermore, there exists a partial composition product: For $\Gamma_1$ in $\mathsf{Gra}(m)$ and $\Gamma_2$ in  $\mathsf{Gra}(n)$ define $\Gamma_1 \circ_i \Gamma_2$ in $\mathsf{Gra}(m+n-1)$ by inserting $\Gamma_2$ for the $i$-th vertex of $\Gamma_1$ and then summing up all possible reconnections to $\Gamma_2$ of edges previously connected to the $i$-th vertex of $\Gamma_1$. The edges are renumbered in a way that puts the edges of $\Gamma_2$ at the end while otherwise keeping the ordering of the edges. For details of the definition, we refer to \cite{willwacher2010m} and \cite{merkulov2010grothendieck}. With the $\mathbb{S}$-action and partial composition just defined, $\mathsf{Gra}$ is an operad.

The $\mathbb{S}$-invariants of $\mathsf{Gra}$ can be seen as graphs with ``unidentifiable'', i.e., no longer numbered, vertices. As explained above, the vector space of $\mathbb{S}$-invariants obtains the structure of a Lie algebra, as does the shifted space $\mathrm{Gra}\{-2\}^\mathbb{S}$, where $\mathrm{Gra}\{-2\}(n)=\mathrm{Gra}(n)[2-2n]$. 
Willwacher denotes this Lie algebra $\mathrm{Gra}\{-2\}^\mathbb{S}$ of graphs with unidentifiable vertices by $\mathrm{fGC}_2$. The subset of graphs whose vertices are at least trivalent forms a sub-Lie algebra denoted by $\mathrm{GC}_2$. One checks that the graph $\MC \in \mathrm{fGC}_2$ satisfies the Maurer-Cartan equation $[\MC,\MC]=0$. Hence it equips $\mathrm{fGC}_2$ with the differential $[\MC, -]$. Furthermore, as proved in e.g. \cite{willwacher2010m}, $\mathrm{GC}_2$ is closed with respect to the differential. We have hence constructed a differential graded Lie algebra $(\mathrm{GC}_2,[-,-], [\MC, -])$. Willwacher's main result in \cite{willwacher2010m} is that the zeroth cohomology of this complex is isomorphic as a Lie algebra to the Grothendieck-Teichm\"uller algebra: $H^0(\mathrm{GC_2}) \cong \grt$.

Let now $\g$ be a graded vector space. We define $\mathrm{CE}(\g)$, the Chevalley-Eilenberg cochain complex of $\g$. Its cochains are given by
\begin{equation*}
\mathrm{CE}(\g) = \prod_{n} \mathrm{Hom}( \mathrm{Sym}^n \g[1], \g[1]) \cong \prod_{n} \mathrm{Hom}({\Lambda^n } \g,\g)[1-n] \cong \mathrm{CoDer}(\mathrm{Sym}\ \g[1]),
\end{equation*}
where $\mathrm{Sym}$ denotes the symmetric algebra and $\mathrm{CoDer}$ the space of coderivations.
The Chevalley-Eilenberg cochains form a graded Lie algebra with the graded Nijenhuis-Richardson bracket $[-,-]_\mathrm{NR}$. It is defined by all possible insertions of homomorphisms into each other, see \cite{MR0214636} and \cite{lecomte1992multigraded}.
Observe that a 1-cochain $\pi_\g \in \mathrm{Hom}^0(\g\wedge \g, \g)$ in the Chevalley-Eilenberg complex is a Lie algebra structure on $\g$ if and only if it satisfies the equation $[\pi_\g,\pi_\g]_\mathrm{NR}$. Furthermore, any Lie bracket $\pi_\g$ on $\g$ yields a differential $d_{\pi_\g} = [\pi_\g,-]_\mathrm{NR}$ on $\mathrm{CE}(\g)$. This cochain complex with the new differential $d_{\pi_\g}$ is called the Chevalley-Eilenberg cochain complex of the Lie algebra $(\g, \pi_\g)$.

We continue by stating some well-known facts that are treated in more detail in e.g. \cite{manetti1998deformation}, \cite{yekutieli2006continuous} or the book \cite{lodayalgebraic}. Given a zero-cocycle $\gamma$, i.e., $ \gamma \in \mathrm{CE}^0(\g)= \mathrm{CoDer}^0(\mathrm{Sym}\ \g[1])$ and $d_{\pi_\g}(\gamma)=0$, its exponential $\mathrm{exp}(\gamma) = \sum_{n=1}^\infty \frac{1}{n!} \gamma^n$ is an $L_\infty$-automorphism of $\g$. Call zero-cocycles $\gamma$, $\gamma'$ gauge-equivalent if $\gamma - \gamma'$ is a coboundary. A well-known theorem states that if $\gamma$ and $\gamma'$ are gauge-equivalent then  $\mathrm{exp}(\gamma)$ and $\mathrm{exp}(\gamma')$ are homotopy-equivalent. Hence the zero-cohomology $H^0(\mathrm{CE}(\g),d_{\pi_\g})$ of the Chevalley-Eilenberg cochain complex of the Lie algebra $(\g,\pi_\g)$ acts on $\g$ by $L_\infty$-automorphisms modulo homotopy equivalence.

\begin{proof}[Proof of Theorem 1]

We start with an outline of the proof. In \cite{willwacher2010m}, Willwacher shows that the Grothendieck-Teichm\"uller algebra $\grt$ is isomorphic as a Lie algebra to the zeroth cohomology $H^0(\gc_2)$ of the Kontsevich graph complex $\gc_2$. Furthermore, as we have stated above, $H^0\mathrm{CE}(\tpoly(\R^d)[1])$ acts on $\tpoly(\R^d)[1]$ by $L_\infty$-automorphisms modulo homotopy equivalence.
Hence it suffices to prove that there is a morphism of Lie algebras $H^0(\gc_2) \rightarrow H^0 \mathrm{CE}(\tpoly(\R^d)[1])$, because it will follow that $\grt \cong H^0(\gc_2)$ acts on $\tpoly(\R^d)[1]$ by $L_\infty$-automorphisms modulo homotopy equivalence.
Observe that we have to shift $\tpoly(\R^d)$ by 1 to turn the odd Lie algebra into a usual graded Lie algebra. 

The plan is to construct a Lie algebra morphism
\begin{equation*}
  \gc_2 \rightarrow \mathrm{CE}(\tpoly(\R^d)[1]) 
\end{equation*}
commuting with the differentials $[\MC,-]$ of $\gc_2$ and $[\pi_S, -]_\mathrm{NR}$ of $\mathrm{CE}(\tpoly(\R^d)[1])$, where $\pi_S$ denotes the shifted Schouten bracket on $\tpoly(\R^d)[1]$. Recall that $\gc_2$ is a subalgebra of the Lie algebra $\mathrm{fGC}_2$ of $\mathbb{S}$-invariants of the operad $\mathsf{Gra}$. We will identify $\mathrm{CE}(\tpoly(\R^d)[1])$ with the Lie algebra of $\mathbb{S}$-invariants of another operad, $\mathrm{End}_{\tpoly(\R^d)[2]}$. An operad morphism between the two operads will induce a Lie algebra morphism between $\mathrm{fGC}_2$ and $\mathrm{CE}(\tpoly(\R^d)[1])$, which restricts to $\gc_2$. This Lie algebra morphism will not commute with the differentials, but the standard trick of twisting with a Maurer-Cartan element will remedy that. As said before, the existence of such a morphism proves the theorem.

The endomorphism operad $\mathrm{End}_{A}$ of a graded vector space $A$ is defined by $\mathrm{End}_A(n) = \mathrm{Hom}(A^{\otimes n},A)$, where partial composition is the insertion of endomorphisms into each other. The Lie algebra of $\mathbb{S}$-invariants of $\mathrm{End}_{A}$ is the Chevalley-Eilenberg cochain complex $\mathrm{CE}(A[-1])$ of $A[-1]$ with the Nijenhuis-Richardson bracket. Hence $\mathrm{CE}(\tpoly(\R^d)[1])$ is the Lie algebra of $\mathbb{S}$-invariants of $\mathrm{End}_{\tpoly(\R^d)[2]}$.

We construct an operad map $\mathsf{Gra} \rightarrow \mathrm{End}_{\tpoly(\R^d)[2]}$ by mapping a graph $\Gamma$ to the endomorphism
\begin{equation} \label{phigamma}
 \Phi_\Gamma( \gamma_1[2], \ldots, \gamma_n[2] ) = \left( \mu \circ \prod_{(i,j) \in E(\Gamma)} 
	\sum_{k=1}^d  \Delta_{(i,j)}(\gamma_1 \otimes \ldots \otimes \gamma_n) \right) [2],
\end{equation}
where $\Delta_{(i,j)}$ is defined by
\begin{equation*}
 \Delta_{(i,j)} = \frac{\partial}{\partial x^k_{(j)}} \frac{\partial}{\partial \psi_k^{(i)}} + 
	\frac{\partial}{\partial \psi_k^{(j)}} \frac{\partial}{\partial x^k_{(i)}}.
 \end{equation*}
Here, $\mu$ is multiplication of polyvector fields, $\frac{\partial}{\partial x^k_{(j)}}$ and $\frac{\partial}{\partial \psi_k^{(j)}}$ are differentiation with respect to $x^k$ or $\psi_k$, respectively, in the $j$-th argument, and $\prod$ is concatenation of differential operators, where the order follows the ordering of the edges. One checks that the map in fact respects operadic composition. Hence it induces a Lie algebra morphism $\mathrm{fGC}_2 \rightarrow \mathrm{CE}(\tpoly(\R^d)[1]) $ which restricts to the subalgebra $\gc_2$. 

Recall that $\mathrm{CE}(\tpoly(\R^d)[1])$ is equipped with the differential $[\pi_S, -]_\mathrm{NR}$, where $\pi_S$ denotes the shifted Schouten bracket, and $\gc_2$ with the differential $[\MC,-]$. The just constructed Lie algebra morphism $\gc_2 \rightarrow \mathrm{CE}(\tpoly(\R^d)[1]) $ does not respect these differentials, only the zero differential. However, it maps the Maurer-Cartan element $\MC$ to the Schouten bracket. 
Hence twisting the morphism with $\MC$ yields a dg-Lie algebra morphism compatible with both differentials. This concludes the proof.
\end{proof}

\begin{cor}
 The Grothendieck-Teichm\"uller group GRT acts on the Schouten algebra $\tpoly(M)$ on a general smooth manifold $M$ by  $L_\infty$-automorphisms defined up to homotopy equivalence. The action depends on the choice of a torsion-free connection on $M$.
\end{cor}
\begin{proof}
We need to check that the $L_\infty$-automorphisms in the image of the morphism
\begin{equation*}
 \grt \cong H^0(\gc_2) \rightarrow H^0 \mathrm{CE}(\tpoly(\R^d)[1]) \rightarrow \mathrm{Aut}(\tpoly(\R^d)[1])/\sim
\end{equation*}
satisfy the four conditions of the main theorem. Recall that we construct a morphism from the graph complex $\gc_2$ to the Chevalley-Eilenberg cochain complex $\mathrm{CE}(\tpoly(\R^d)[1])$. The image in $\mathrm{CE}^0(\tpoly (\R^d)[1])$ of a degree zero element $\Gamma \in \gc_2$ is a degree zero map $\Psi_\Gamma: \mathrm{Sym}^n (\tpoly (\R^d)[2]) \rightarrow \tpoly (\R^d)[2]$. We have to show that its exponential satisfies the four conditions of the main theorem. Observe that it suffices to show that $\Psi_\Gamma$ satisfies the conditions. The important fact is that the graphs in $\gc_2$ have the property that each vertex is at least trivalent.
\subsubsection*{Condition \ref{Rformal}}
The formula \eqref{phigamma} works for $\gamma_1, \ldots, \gamma_n$ in both $\tpoly (\R^d)$ and $\tpoly( \R^d_\mathrm{formal})$. It follows that the whole construction of $\Psi_\Gamma$ runs through for $\tpoly (\R^d_\mathrm{formal})$ as well.
\subsubsection*{Condition \ref{changeofcoordinates}}
Because $\Delta_{(i,j)}$ is invariant under linear change of coordinates of $\R^d_\mathrm{formal}$, this also holds for $\Phi_\Gamma$ and $\Psi_\Gamma$.
\subsubsection*{Condition \ref{vectorfields}}
 It suffices to show that $\Phi_\Gamma$ as defined in \eqref{phigamma} vanishes on vector fields for any graph $\Gamma \in \mathsf{Gra}$. Denote by $\mathfrak{D}(\Gamma)$ the set of all directed graphs that can be obtained by giving a direction to each edge in $\Gamma$. We can then rewrite the definition of $\Phi_\Gamma$ as
\begin{equation*}
 \Phi_\Gamma( \gamma_1[2], \ldots, \gamma_n[2] ) =  \left( \mu \circ \sum_{\Gamma' \in \mathfrak{D}(\Gamma)} \prod_{(i,j) \in E(\Gamma')} 
	\sum_{k=1}^d \left( \frac{\partial}{\partial x^k_{(j)}} \frac{\partial}{\partial \psi_k^{(i)}}  \right)(\gamma_1 \otimes \ldots \otimes \gamma_n) \right)[2].
\end{equation*}
As the vertices of $\Gamma$ are at least trivalent, for any directed $\Gamma'$ in $\mathfrak{D}(\Gamma)$ there is a vertex $l$ that has at least two outgoing edges. This means that we differentiate $\gamma_l$ at least twice with respect to $\psi$. As $\gamma_l$ is a vector field, this is zero. Hence the summand belonging to $\Gamma'$ is zero as well. As we can find such a vertex for any choice of $\Gamma'$, it follows that $\Phi_\Gamma$ is zero on vector fields.
\subsubsection*{Condition \ref{linearvectorfield}}
It suffices to show that $\Phi_\Gamma$ is zero if one of the inputs is a vector field that is linear in the standard coordinates on $\mathbb{R}^d$. We use the same rewriting of the definition of $\Phi_\Gamma$ as before. Pick a $\Gamma'$ in $\mathfrak{D}(\Gamma)$ and assume $\gamma_l$ is a vector field linear in the coordinates of $\R^d$. As the vertex $l$ is at least trivalent, it has at least two outgoing edges or at least two ingoing edges. Hence we differentiate the vector fields $\gamma_l$ at least twice with respect to the $x$ or with respect to the $\psi$, both of which yields zero. It follows that the summand belonging to this $\Gamma'$ is zero. As this can be shown for any choice of $\Gamma'$, we have showed that $\Phi_\Gamma$ is zero if one of the inputs is a vector field that is linear in the standard coordinates on $\mathbb{R}^d$.
\end{proof}

\section*{Acknowledgments}

I would like to thank Sergei Merkulov for proposing this problem to me and for much help. I also had the honor of getting help from Torsten Ekedahl. The anonymous referee contributed with very helpful advice, especially the remark how the dependence on the choice of a connection might be removed. Finally, thanks to my fellow students Johan Alm, Theo Backman, Johan Granåker and Henrik Strohmayer for all the discussions on the many questions I had.

\bibliographystyle{amsplain}
\bibliography{glob}{}

\end{document}